\documentclass[12pt, reqno]{amsart}
\usepackage{amsmath}
\usepackage{amssymb}
\usepackage{epsfig}
\usepackage{graphicx}
\usepackage{color}
\definecolor{shadecolor}{gray}{0.875}
\usepackage{amscd}
\usepackage{comment}
\usepackage{adjustbox}
\usepackage{multirow}
\usepackage{tikz-cd}
\usepackage[all]{xy}

\numberwithin{equation}{section}

\usetikzlibrary{positioning}
\usepackage[colorlinks=false,urlbordercolor=white]{hyperref}
  
\tikzset{sgplattice/.style={inner sep=1pt,norm/.style={red!50!blue},char/.style={blue!50!black},
  lin/.style={black!50}},cnj/.style={black!50,yshift=-2.5pt,left=-1pt of #1,scale=0.5,fill=white}}

\DeclareFontFamily{U}{mathb}{\hyphenchar\font45}
\DeclareFontShape{U}{mathb}{m}{n}{
      <5> <6> <7> <8> <9> <10> gen * mathb
      <10.95> mathb10 <12> <14.4> <17.28> <20.74> <24.88> mathb12
      }{}
\DeclareSymbolFont{mathb}{U}{mathb}{m}{n}
\DeclareMathSymbol{\righttoleftarrow}{3}{mathb}{"FD}

\theoremstyle{plain}
\newtheorem{prop}{Proposition}[section]

\newtheorem{theo}[prop]{Theorem}
\newtheorem{coro}[prop]{Corollary}

\newtheorem{lemm}[prop]{Lemma}

\theoremstyle{definition}
\newtheorem{defi}[prop]{Definition}
\newtheorem{exam}[prop]{Example}

\newtheorem{rema}[prop]{Remark}

\newcommand{\eqto}{\stackrel{\lower1.5pt\hbox{$\scriptstyle\sim\,$}}\to}
\newcommand{\eqdashto}{\stackrel{\lower1.5pt\hbox{$\scriptstyle\sim\,$}}\dashrightarrow}

\newcommand{\actsfromright}{\righttoleftarrow}

\def\cL{{\mathcal L}}

\def\cO{{\mathcal O}}

\def\bA{{\mathbb A}}

\def\bP{{\mathbb P}}

\def\bZ{{\mathbb Z}}

\def\rH{{\mathrm H}}

\def\Br{\mathrm{Br}}
\def\codim{\mathrm{codim}}

\def\CH{\mathrm{CH}}

\def\Pic{\mathrm{Pic}}
\def\End{\mathrm{End}}

\def\Am{\mathrm{Am}}

\def\lim{\mathrm{lim}}

\def\IJ{\mathrm{IJ}}

\def\Spec{\mathrm{Spec}}

\author{Andrew Kresch}
\address{
  Institut f\"ur Mathematik,
  Universit\"at Z\"urich,
  Winterthurerstrasse 190,
  CH-8057 Z\"urich, Switzerland
}
\email{andrew.kresch@math.uzh.ch}

\author{Yuri Tschinkel}
\address{Courant Institute\\
                New York University \\
                New York, NY 10012 \\
                USA }
\address{Simons Foundation\\
                 160 Fifth Av.\\ 
                 New York, NY 10010}                
\email{tschinkel@cims.nyu.edu}

\title[Equivariant birational geometry]{Linearizability notions in equivariant birational geometry}

\begin{document}

\date{June 8, 2026}

\begin{abstract}
We discuss birational properties of actions of finite groups on algebraic varieties, linearizability, torsors, and versality. 
\end{abstract}

\maketitle

\section{Introduction}
\label{sect:intro}

In this note, we continue our study of equivariant birational geometry of regular actions of finite groups on smooth projective varieties
over an algebraically closed field of characteristic zero. We amplify and expand the parallels between equivariant geometry and geometry over nonclosed fields, extending constructions from \cite{DR} to the setting of equivariant Chow groups, and equivariant decomposition of the diagonal. One of the key observations in \cite{DR} was that the twisting construction preserves some birational properties of $G$-actions, e.g., a $G$-variety is equivariantly stably linearizable if and only if all of its twists are stably rational.

Among our main results are:
\begin{itemize}
\item In Theorem~\ref{thm.onetorsorP} we show that many properties of arbitrary twists of $G$-varieties via torsors (such as existence of rational points, unirationality, stable rationality, rationality) are implied by the same properties for a \emph{single} twist, arising from a $G$-variety with generically free versal $G$-action.  
\item In Theorem~\ref{thm.onetorsorXY} we show that birationality of twists of $G$-varieties 
for all torsors is equivalent to birationality of twists arising from a single generically free versal $G$-action.   
\item In Theorem~\ref{thm.idodtwist} we prove that $G$-equivariant integral decomposition of the diagonal, introduced in Section~\ref{sect:obstr}, is equivalent to integral decomposition of the diagonal for a twist arising from a single generically free versal actions.  
\end{itemize}
Earlier developments have exhibited properties of arbitrary twists as implied by the same for a single twist, associated with a faithful $G$-representation. 
As linearizable actions are versal, but not all versal actions are linearizable,
for these properties our statements represent a strengthening.

In Section~\ref{sect:gen} we discuss basic notions in $G$-equivariant birational geometry and their relation to birational geometry over nonclosed fields, via the twisting construction 
through torsors. In Section~\ref{sec.coarse} we introduce the notions of {\em coarse $G$-birationality}, which sits between $G$-birationality and stable $G$-birationality, and \emph{coarse linearizability}. In Section~\ref{sect:obstr} we define the notion of $G$-equivariant decomposition of the diagonal and establish its key properties, e.g., connection with 
stable equivariant linearizability, in Theorem~\ref{theo:stablelinearizable},  and specialization, in Theorem~\ref{theo:specialization}. Finally, in Section~\ref{sect:coho}, we  
consider the behavior of cohomological obstructions under twisting, and the related notion of {\em negligible cohomology}.   

\medskip
\noindent
\textbf{Acknowledgments:}
The authors thank Brendan Hassett and Sho Tanimoto for discussions and collaboration on related problems.
The second author was partially supported by NSF grant 2301983.

\section{Basic notions of equivariant birational geometry}
\label{sect:gen}
We work over an algebraically closed field $k$ of characteristic zero.

\subsection*{Varieties over extension fields}
For an extension field $K$ of $k$,
a variety over $K$ is a scheme of finite type over $K$ that is separated, reduced and irreducible.
A proper birational map $X\dashrightarrow Y$ of varieties over $K$ is one that can be factored into proper birational morphisms and their formal inverses.
Existence of a proper birational map over $K$ defines an equivalence relation, for which we employ the notation
\[ X\sim Y. \]
For the stable version, with $X$ replaced by $X\times \bP^m$ and $Y$ by $Y\times \bP^n$, for some $m$ and $n$, we employ the notation
\[ X\sim^s Y. \]
%Two varieties $X$ and $Y$ over $K$ are birationally equivalent if $K(X)\cong K(Y)$, and stably birationally equivalent if $X\times \bP^m$ and $Y\times\bP^n$ are birationally equivalent for some $m$ and $n$. 
We recall well-studied conditions for a variety $X$ over $K$ of dimension $d$ (understood for general $X$ as applied to a proper model of $X$):
\begin{itemize}
\item {\em rationality}: 
$X\sim\bP^d$, 
\item {\em stable rationality}: $X\sim^s\bP^d$.
\end{itemize}

\subsection*{$G$-varieties}
Let $G$ be a finite group.
A $G$-variety (over $k$) is a separated, reduced, finite-type $k$-scheme $X$ with regular $G$-action, such that the $G$-action is transitive on the set of irreducible components of $X$.
We call $X$ an irreducible $G$-variety if, non-equivariantly, $X$ is irreducible, i.e., a variety.
Given a $G$-variety $X$, the $G$-action is generically free if it is free on an invariant dense open subvariety of $X$.
By convention, the actions are right $G$-actions.

We recall some further notions; see, e.g., \cite[\S 3.1]{KTsurvey} for more details.
A $G$-rational map is a $G$-equivariant rational map between $G$-varieties.
A $G$-birational map is a $G$-rational map that is birational.
A proper $G$-birational map is one that factors as a composition of proper birational $G$-equivariant morphisms and their formal inverses.

Among $G$-varieties we have the relation, to be \emph{$G$-birational}:
\[ X\sim_GY\quad\Leftrightarrow\quad
\exists\text{ proper $G$-birational map $X\dashrightarrow Y$}. \]
We call $X$ and $Y$ \emph{stably $G$-birational} if they become $G$-birational after passing to products with (nonempty) projectivized representations:
\[ X\sim_G^sY\quad \Leftrightarrow\quad
\exists\text{ $G$-representations $U$, $V$}\colon X\times \bP(U)\sim_G Y\times \bP(V). \]
We also recall birationally invariant conditions on a $G$-variety $X$:
\begin{itemize}
\item \textbf{(L)} \emph{linearizability}: $X\sim_G\bP(V)$ for some representation $V$.
\item \textbf{(SL)} \emph{stable linearizability} $X\times \bP(U)\sim_G\bP(V)$ for some representations $U$ and $V$.
\item \textbf{(U)} \emph{$G$-unirationality} if there is a dominant $G$-rational map $\bP(V)\dashrightarrow X$ for some representation $V$.
\end{itemize}
(These conditions apply as stated when $X$ is proper and are understood by convention for more general $X$ when they are satisfied by an equivariant proper model; cf.\ \cite[\S 3.1]{KTsurvey}.)

\subsection*{Torsors}
Further notions are defined in terms of torsors, by which we mean torsors under $G$.
Without further mention, the base of the torsor will be $\Spec(K)$, where $K$ is an extension field of $k$.
The total space will be $T=\Spec(L)$ for an \'etale $K$-algebra $L$.
The next condition on a $G$-variety $X$ is defined by requiring a condition to hold for all torsors, over all extension fields $K$ of $k$ (expressed simply as $\forall\,T\colon\dots$).
\begin{itemize}
\item \textbf{(WV)} $X$ is \emph{weakly versal} if $\forall\,T$: $\exists$ $G$-equivariant $\Spec(L)\to X$.
\item \textbf{(V)} $X$ is \emph{versal} if every invariant dense open subvariety of $X$ is weakly versal.
\end{itemize}

Let $Z$ be a quasi-projective $G$-variety over $k$ with a generically free action of $G$ and $Z\to Z/G$ the quotient.
The function field $k(Z/G)$ is the field of $G$-invariants $k(Z)^G$.
We get an associated torsor
\[ T_Z:=k(Z)/k(Z)^G, \]
corresponding to the Galois field extension,
with Galois group $G$.
An important special case $T_V$ arises from a faithful $G$-representation $V$ and the quotient $V\to V/G$.

\subsection*{Twisting}
Given a torsor $T=\Spec(L)$ over $K$, we obtain a free $G$-action on the $G$-variety $$
X_L=X\times_{\Spec(k)}\Spec(L)
$$ 
over $K$.
The quotient variety for this free $G$-action, provided it exists, is the \emph{twist} ${}^TX$,
a variety over $K$ with the property
$({}^TX)_L\cong X_L$.
The quotient variety is guaranteed to exist when $X$ is quasi-projective.
In the sequel we will apply the twisting construction only to quasi-projective varieties.

When applied to a torsor $T_Z$ arising from a generically free 
$G$-action on $Z$ as above, we get the twist ${}^{T_Z}X$ over $K=k(Z)^G$.

\begin{exam}
\label{exa.Vtwist}
For a $G$-representation $V$ of dimension $n\ge 1$, we have
${}^TV\cong \bA^n$ \cite[Lemma 3.1]{DR} and
${}^T\bP(V)\cong \bP^{n-1}$, for all torsors $T$.
\end{exam}

\begin{rema}
    \label{rema:twist-stab}
The twisting construction respects $G$-invariant loci $U\subset X$, i.e.,
${}^TU\subset {}^TX$, for all $G$-torsors $T$. In particular, it preserves the stabilizer stratification on $X$, see \cite[Sect.\ 6.1]{KTsurvey}.
\end{rema}

\subsection*{Versality and twisting}

Duncan and Reichstein give the following characterizations.

\begin{prop}[{\cite{DR}}]
\label{prop.DR}
Let $X$ be a quasi-projective $G$-variety.
\begin{itemize}
\item[(i)] $X$ is weakly versal $\Leftrightarrow\,\forall\,T\colon{}^TX(K)\ne \emptyset$.
\item[(ii)] $X$ is versal $\Leftrightarrow\,\forall\,T\colon{}^TX(K)$ is dense in ${}^TX$.
\item[(iii)] $X$ is $G$-unirational $\Leftrightarrow\,\forall\,T\colon{}^TX$ is unirational.
\item[(iv)] $X$ is stably linearizable $\Leftrightarrow\,\forall\,T\colon{}^TX$ is stably rational.
\end{itemize}
\end{prop}

For smooth projective $G$-varieties, these conditions are stably birationally invariant.
Generally, we have the implications \cite[(1.1)]{DR}:
\[ \textbf{(SL)}\quad\Rightarrow\quad \textbf{(U)}\quad\Rightarrow\quad \textbf{(V)}\quad\Rightarrow\quad \textbf{(WV)}. \]
(In \cite{DR}, group actions satisfying \textbf{(U)} are called \emph{very versal}.) 
There is also the implication
\[ \textbf{(WV)}\quad\Rightarrow\quad \textbf{(A)} \]
\cite[Rmk.\ 2.7]{DR},
where Condition \textbf{(A)} is the existence of fixed points for all abelian subgroups of $G$.

\begin{exam}
\label{exa.WVV}
All of these implications are strict.
For the first two, we see this already when
$G$ is trivial \cite[Sect.\ 1]{DR}.
For $G=\bZ/2\bZ$, the action on an elliptic curve $E$ by $-1$ has fixed points, thus is weakly versal \cite[Prop.\ 2.2]{DR}, but is not versal:
for $T\colon k(t)/k(t^2)$, the only rational points on the twist ${}^TE$ are the ones coming from the fixed points.
Actions satisfying condition \textbf{(A)}, but not \textbf{(WV)}, may be found among toric varieties and del Pezzo surfaces \cite{KT-uni}, \cite{STZ}.
\end{exam}

In the proof of their characterizations by means of torsors, Duncan and Reichstein make the observation that a property for all torsors is implied by the property for a single torsor, coming from a faithful $G$-representation.
Such observations had also appeared earlier, e.g., as recalled in \cite[Prop.\ 2.1]{GM}.
Our next result presents this observation uniformly for all of the relevant properties and in the strengthened form, that the single torsor can be given by any generically free \emph{versal} group action.

\begin{theo}
\label{thm.onetorsorP}
Let $X$ be a quasi-projective $G$-variety over $k$, and
$\mathbf{P}$ one of the following properties of a variety over an extension field of $k$: 
\begin{itemize}
    \item possession of a rational point,
\item possession of a Zariski dense set of rational points,
\item unirationality,
\item stable rationality,
\item rationality.
\end{itemize}
The following are equivalent.
\begin{itemize}
\item[(i)] For every torsor $T$, the twist
${}^TX$ satisfies property $\mathbf{P}$.
\item[(ii)] For some quasi-projective $G$-variety $Z$ over $k$ with generically free versal $G$-action and 
$T_Z=k(Z)/k(Z)^G$, 
the twist ${}^{T_Z}X$ 
%${}^{k(Z)/k(Z)^G}X$ 
satisfies property $\mathbf{P}$.
\end{itemize}
\end{theo}

\begin{proof}
The implication (i) $\Rightarrow$ (ii) is trivial.
We suppose (ii), and let
$K$ be an extension field of $k$ and
$T\colon L/K$ a $G$-torsor, where
$L$ is an \'etale $K$-algebra.
We need to establish property $\mathbf{P}$
for the twist ${}^TX$.
We first treat the case when $\mathbf{P}$ is ``possession of a rational point'', then we explain the modifications to the argument for
the other properties $\mathbf{P}$.
If ${}^{T_Z}X$ has a rational point, then there is a $G$-rational map $Z\dashrightarrow X$, say, defined on
invariant dense open $U\subset Z$.
Since the $G$-action on $Z$ is versal, there is a $G$-equivariant morphism $\Spec(L)\to U$.
Composing, we get equivariant $\Spec(L)\to X$, thus a $K$-point of ${}^TX$.

Generally, density of a given set of points of a finite-type scheme over a field is stable under field extension (reduce to the case of a finitely generated field extension and treat the finite and purely transcendental cases separately).
So, if ${}^{T_Z}X$ has a dense set of $k(Z)^G$-points, then
${}^{K(Z)/K(Z)^G}X$ has a dense set of $K(Z)^G$-points.
For invariant dense open $W\subset X_K$ we argue as above with $Z_K\dashrightarrow W$, defined on invariant dense open $U\subset Z_K$, and get a $K$-point of ${}^TX$, for which the corresponding $G$-equivariant morphism $\Spec(L)\to X_K$ has image in $W$.

If ${}^{T_Z}X$ is unirational, then for some $N$ there exists a dominant $G$-rational map $\bP^N\times Z\dashrightarrow X\times Z$ (with trivial action on $\bP^N$), compatible with the projection morphisms to $Z$.
We choose an invariant dense open subset where this is a morphism and an invariant dense open subset $W$ of the image and let $U\subset Z$ denote the image of $W\subset X\times Z$ under projection.
As before, by versality, there is equivariant $\Spec(L)\to U$, and by base change, dominant equivariant $\bP^N_L\dashrightarrow X_L$, hence dominant $\bP^N_K\dashrightarrow {}^TX$.

For the remaining properties, we get the result as a special case of the next theorem.
\end{proof}

\begin{theo}
\label{thm.onetorsorXY}
Let $X$ and $Y$ be a quasi-projective $G$-varieties over $k$.
The following are equivalent.
\begin{itemize}
\item[(i)] For every torsor $T$, we have
${}^TX\sim {}^TY$.
\item[(ii)] For some quasi-projective $G$-variety $Z$ over $k$ with generically free versal $G$-action, we have ${}^{T_Z}X\sim {}^{T_Z}Y$.
\end{itemize}
\end{theo}

\begin{proof}
As before, (i) $\Rightarrow$ (ii) is trivial, so we suppose (ii) and consider a $G$-torsor $T\colon L/K$.
By a closure-of-graph construction we have, after replacing $Z$ by a suitable invariant dense open subvariety, a diagram of equivariant  projective birational morphisms
\[
\xymatrix{
&V \ar[dl] \ar[dr] \\
X\times Z && Y\times Z
}
\]
compatible with projection morphisms to $Z$.
We let $W\subset V$ be an invariant dense open subset, such that the morphisms in the diagram restrict to isomorphisms of $V$ with the respective images.
Let $U\subset Z$ to be the common image under the further projection morphisms.
We conclude as in the proof of Theorem \ref{thm.onetorsorP} with an equivariant morphism $\Spec(L)\to U$ and base change.
\end{proof}

\section{Coarse $G$-birationality and coarse linearizability}
\label{sec.coarse}

Theorem \ref{thm.onetorsorXY} motivates the following definition.

\begin{defi}
\label{def.coarsebirationality}
Two quasi-projective $G$-varieties $X$ and $Y$ are said to be \emph{coarsely $G$-birational}, written
\[ X\sim_G^c Y \]
if the equivalent conditions of Theorem \ref{thm.onetorsorXY} hold for $X$ and $Y$.
\end{defi}

\begin{exam}
\label{exa.cbrep}
We have $\bP(V)\sim_G^c\bP^{n-1}$ for any $G$-representation $V$ of dimension $n\ge 1$ (cf.\ Example \ref{exa.Vtwist}).
\end{exam}

We recall that the case $Y=\bP^d$ of Theorem \ref{thm.onetorsorXY}, where $d$ denotes the dimension of $X$, gives
Theorem \ref{thm.onetorsorP} for the property rationality.

\begin{defi}
\label{def.CL}
We say that a projective $G$-variety $X$ is \emph{coarsely linearizable}, or satisfies Condition \textbf{(CL)}, if $X$ satisfies the equivalent conditions of Theorem \ref{thm.onetorsorP} for the property rationality.
Generally, a $G$-variety considered to satisfy Condition \textbf{(CL)}, when \textbf{(CL)} holds for an equivariant projective model.
\end{defi}

In other words, for a projective $G$-variety $X$ of dimension $d$:
\[ \text{$X$ is coarsely linearizable}\quad\Leftrightarrow\quad X\sim_G^c \bP^d. \]

\begin{exam}
\label{exa.dimension}
The obstruction to linearizability, based on dimension of $G$-representations \cite[\S 4.1]{KTsurvey}, yields examples of actions that are coarsely linearizable but not linearizable.
Letting the symmetric group $\mathfrak{S}_3$ act faithfully on $\bP^1$, the product action of $\mathfrak{S}_3\times \mathfrak{S}_3$ on $\bP^1\times \bP^1$ is coarsely linearizable but not linearizable: the smallest faithful representation of $\mathfrak{S}_3\times \mathfrak{S}_3$ has dimension $4$.
For additional examples, notice that
\textbf{(CL)} $\Leftrightarrow$ \textbf{(SL)}
for any nonsingular invariant quadric $X\subset \bP(V)$.
So a stably linearizable dimension-obstructed quadric, as given in \cite[\S 4.1]{KTsurvey}, is coarsely linearizable but not linearizable.
\end{exam}

\begin{exam}
\label{exa.alltwistsrational}
Some varieties are known to have the property that all twists are rational, or at least all twists with a rational point are rational.
Examples include projective spaces, quadric hypersurfaces, 
del Pezzo surfaces of degree $\ge 5$, 
and quintic del Pezzo threefolds and prime Fano threefolds of genus 7 and 12 \cite[Thm.\ 1.1]{Kuznetsov-Prokhorov}.
On the other hand, methods such as universal torsors
or vector bundle techniques
give the stable linearizability of some nonlinearizable actions.
Combining, we get additional examples of actions that are coarsely linearizable but not linearizable.
These include a sextic del Pezzo surface with $\mathfrak{S}_4$-action \cite[Sect.\ 5.2]{HT23},
$\overline{\mathcal{M}}_{0,2n+1}$ for many subgroups of $\mathfrak{S}_{2n+1}$ \cite{HTZ},  and
some actions on quadrics \cite[Prop.\ 12]{BBT}.
\end{exam}

\begin{rema}
\label{rem.LCLSL}
Condition \textbf{(CL)} fits into a chain of implications
\[ \textbf{(L)}\quad\Rightarrow\quad \textbf{(CL)}\quad\Rightarrow\quad \textbf{(SL)}. \]
As we see in Examples \ref{exa.dimension} and \ref{exa.alltwistsrational}, the
first implication cannot be reversed.
When $G$ is the trivial group, \textbf{(L)} and \textbf{(CL)} are both equivalent to rationality, and \textbf{(SL)}, to stable rationality.
Since there are
known examples of stably rational varieties that are not rational \cite{BCTSSD}, also the second implication cannot be reversed.
\end{rema}

\begin{exam}
\label{exa.Q8Q8}
We let the quaternion group $\mathfrak{Q}_8$ act linearly on $\bP^1$, via faithful action of the $\mathfrak{K}_4$-quotient, with center $Z\subset \mathfrak{Q}_8$ acting trivially.
The product action of $G=\mathfrak{Q}_8\times \mathfrak{Q}_8$ on $X=\bP^1\times \bP^1$ has trivial action of $Z\times Z$.
Since any $\bP(V)$, $\dim(V)=3$, has trivial action of a subgroup of $G$ of order $\ge 8$, the action is not linearizable.
Being a product of linear actions, it is coarsely linearizable.
Thus it is also stably linearizable.
It is also interesting to note that $X\times \bP^n$ is not linearizable for any $n$, where $G$ acts trivially on the factor $\bP^n$.
Indeed, if $X\times \bP^n\sim_G \bP(V)$,
then $Z\times Z$ must act trivially on $\bP(V)$, thus $X\times\bP^n\sim_{\mathfrak{K}_4\times \mathfrak{K}_4}\bP(V)$.
This would contradict the
equivariant birational invariance of the Amitsur group $\mathrm{Am}^2$; cf.\ \cite[Sect.\ 5]{KTsurvey}:
$$
\mathrm{Am}^2(X\times \bP^n,\mathfrak{K}_4\times \mathfrak{K}_4)\cong \bZ/2\bZ\oplus \bZ/2\bZ,
$$
but any projective space has cyclic $\mathrm{Am}^2$.
\end{exam}

\begin{rema}
\label{rem.SLwhy}
Example \ref{exa.Q8Q8} illustrates why \textbf{(SL)} needs to be defined by the condition $X\times \bP(U) \sim_G \bP(V)$, rather than the linearizability of $X\times \bP^n$ for some $n$.
For \emph{generically free} actions, the two formulations are equivalent, since then we have
$X\times \bP(U) \sim_G X\times \bP^n$, $n=\dim(U)-1$,
by the No-Name Lemma; cf.\ \cite[\S 3.5]{KTsurvey}.
\end{rema}

\begin{prop}
\label{prop.implicationsXY}
For projective $G$-varieties $X$ and $Y$ we have:
\[ X\sim_GY\quad\Rightarrow\quad X\sim_G^cY\quad\Rightarrow\quad X\sim_G^sY. \]
\end{prop}

\begin{proof}
The first implication is clear.
For the second, we take $V$ to be a $G$-representation, such that the $G$-action on $\bP(V)$ is generically free,
and apply Theorem \ref{thm.onetorsorXY} to the versal $G$-action on $Z=\bP(V)$.
A birational equivalence
${}^{T_Z}X\sim {}^{T_Z}Y$ gives an equivariant birational equivalence
$X\times \bP(V)\sim_G Y\times \bP(V)$.
\end{proof}

\begin{rema}
\label{rem.SL}
Condition \textbf{(SL)} may be formulated equivalently with (compactifications of) representations, rather than projectivized representations, as is done, for instance, in \cite[Sect.\ 1]{DR}.
We summarize other reformulations here.
Let $U$ and $V$ be $G$-representations of respective dimensions $d$, $e\ge 1$, and $W$ a $G$-representation, such that the $G$-action on $\bP(W)$ is generically free.
Suppose that
$$
X\times \bP(U)\sim_G Y\times \bP(V).
$$
Then the same holds after passing to products with $\bP(W)$.
The No-Name Lemma gives 
$$
\bP(U)\times \bP(W)\sim_G \bP(1^d\oplus W) \quad \text{  and } \quad 
\bP(V)\times \bP(W)\sim_G \bP(1^e\oplus W);
$$
cf.\ \cite[Cor.\ 3.3]{KTsurvey}).
So
$$
X\times \bP(1^d\oplus W)\sim_G Y\times \bP(1^e\oplus W).
$$
In particular, $X$ is stably linearizable if and only if 
$$
X\times \bP(1^d\oplus W)\sim_G \bP(1^e\oplus W),
$$
for some $d$, $e\ge 1$.
Also, if $X$ and $Y$ have the same dimension, then there is no loss of generality in taking $U=V$ in the definition of $X\sim_G^sY$.
\end{rema}

\begin{exam}
\label{exa.implicationsXPV}
Taking $Y$ to be $\bP(V)$ in Proposition \ref{prop.implicationsXY} recovers the chain of implications among the notions of linearizability.
By definition, \textbf{(L)} is the condition $X\sim_G\bP(V)$ for some $V$.
By Example \ref{exa.cbrep}, \textbf{(CL)} is equivalent to the condition $X\sim_G^c \bP(V)$ for some $V$.
Arguing as in Remark \ref{rem.SL},
\textbf{(SL)} is equivalent to $X\sim_G^s\bP(V)$ for some $V$.
\end{exam}

Conditions \textbf{(U)}, \textbf{(SL)}, and \textbf{(L)} are defined via $G$-equivariant geometry, while
\textbf{(WV)}, \textbf{(V)}, and the new Condition \textbf{(CL)} are defined in terms of torsors. We now give interpretations in terms of equivariant geometry.

\begin{prop}
\label{prop.WV}
Let $X$ be a $G$-variety.
Let $W$ be a $G$-representation, such that $G$ acts generically freely on $\bP(W)$.
The following are equivalent.
\begin{itemize}
\item[(i)] The $G$-action on $X$ is weakly versal.
\item[(ii)] For some $G$-representation $V$ of positive dimension, there is a $G$-rational map $\bP(V)\dashrightarrow X$.
\item[(iii)] There is a $G$-rational map $\bP(W)\dashrightarrow X$.
\end{itemize}
\end{prop}

\begin{proof}
In the proof we use the following form of Theorem \ref{thm.onetorsorP}, which makes no reference to twists:
\[ \text{\textbf{(WV)}}\quad \Leftrightarrow \quad \exists\text{ $G$-rational }Z\dashrightarrow X \]
(where $Z$ is as in Theorem \ref{thm.onetorsorP}).
The implication (iii) $\Rightarrow$ (ii) is trivial.
Given (ii), there is also a $G$-rational map $\bP(V)\times \bP(W)\dashrightarrow X$; so we obain Condition \textbf{(WV)} by taking $Z=\bP(V)\times \bP(W)$.
Taking $Z=\bP(W)$, we get
(i) $\Rightarrow$ (iii).
\end{proof}

The next statement makes reference to the image closure of a rational map, which is the closure of the image of any dense open on which the map is defined.

\begin{prop}
\label{prop.V}
Let $X$ be a $G$-variety.
Let $W$ be a $G$-representation, such that $G$ acts generically freely on $\bP(W)$.
The following are equivalent.
\begin{itemize}
\item[(i)] The $G$-action on $X$ is versal.
\item[(ii)] The union of image closures of $G$-rational maps $\bP(V)\dashrightarrow X$, over all
$G$-representations $V$, is dense in $X$.
\item[(iii)] The union of image closures of $G$-rational maps $\bP(W)\dashrightarrow X$ is dense in $X$.
\end{itemize}
\end{prop}

\begin{proof}
This is immediate from Proposition \ref{prop.WV}.
\end{proof}

\begin{prop}
\label{prop.coarselybirational}
Let $X$ and $Y$ be quasi-projective $G$-varieties.
Let $W$ be a $G$-representation, such that $G$ acts generically freely on $\bP(W)$.
Then $X$ and $Y$ are coarsely $G$-birational,
if and only if there exists a proper $G$-birational map $X\times \bP(W)\dashrightarrow Y\times \bP(W)$, fitting in a commutative diagram
\[
\xymatrix{
X\times \bP(W) \ar@{-->}[rr]\ar[dr] && Y\times \bP(W)\ar[dl] \\
& \bP(W)
}
\]
with the projection morphisms to $\bP(W)$.
\end{prop}

\begin{proof}
We apply Theorem \ref{thm.onetorsorXY} and use that a proper birational map of twists ${}^{T_Z}X\sim {}^{T_Z}Y$ comes from a proper $G$-birational map of products with $Z$, compatible with the projection maps to $Z$.
\end{proof}

\begin{coro}
\label{cor.CL}
Let $X$ be a projective $G$-variety of dimension $d$.
Let $W$ be a $G$-representation, such that $G$ acts generically freely on $\bP(W)$.
The following are equivalent.
\begin{itemize}
\item[(i)] The $G$-action on $X$ is coarsely linearizable.
\item[(ii)] There exists a $G$-birational map $X\times \bP(W)\dashrightarrow \bP^d\times \bP(W)$, that is compatible with the projection morphisms to $\bP(W)$.
\end{itemize}
\end{coro}

\begin{rema}
\label{rem.manyconditions}
The implications, among the conditions on smooth projective $G$-varieties mentioned so far, are displayed here in one chain:
\[ \textbf{(L)}\ \Rightarrow\ 
\textbf{(CL)}\ \Rightarrow\ 
\textbf{(SL)}\ \Rightarrow\ \textbf{(U)}\ \Rightarrow\ \textbf{(V)}\ \Rightarrow\ \textbf{(WV)}\ \Rightarrow\ \textbf{(A)}. \]
The strictness of these implications is explained, for the first two implications, in Remark \ref{rem.LCLSL}, and
for the remaining ones, in Example \ref{exa.WVV}.
\end{rema}

\section{Decomposition of the diagonal}
\label{sect:obstr}

\subsection*{Equivariant Chow groups}
%We recall equivariant Chow groups introduced in \cite{EG98}.
Let $X$ be a quasi-projective $G$-variety over $k$, of dimension $n$. We recall the definition of 
equivariant Chow groups
\[
\CH_i^G(X) := \CH_{i + \ell}((X\times U)/G),
\]
with diagonal action of $G$ on $X$ and on
$U$, a $G$-invariant Zariski open subset of an
$\ell$-dimensional $G$-representation $V$, such that $G$ acts freely on $U$ and $\codim(V \setminus U)>n-i$; see \cite{EG98}. 
This does not depend on $V$ and $U$. 

An $i$-dimensional $G$-invariant cycle on $X$, as well as 
an $(i + \ell)$-dimensional $G$-invariant cycle on $X \times V$, define classes 
in $\CH_i^G(X)$.
Taking $X$ to be a point yields
\[ \mathrm{CH}^G_i:=\mathrm{CH}^G_i(\Spec(k)), \]
the case treated in \cite{totaro}.
Analogous to the degree map of a proper variety, when $X$ is projective we have
\[ \deg_i^G\colon \CH_i^G(X)\to \CH_i^G. \]
For $i=0$ we have $\CH_i^G=\bZ$ and thus integer-valued $\deg_0^G$, as in the non-equivariant case.

\begin{lemm}\cite[Proposition 1]{EG98}
\label{lemma:G-invariantrepresentability}
If $\alpha \in \CH^G_i(X)$, then there exists an $\ell$-dimensional $G$-representation $V$, and an $(i+\ell)$-dimensional $G$-invariant cycle $Z=\sum a_i[Z_i]$ on $X \times V$, that represents the class of $\alpha$.
\end{lemm}

Finally, over a $k$-point of $U/G$, the fiber of 
\[
(X\times U)/G \to U/G
\]
is isomorphic to $X$.
Pullback to a fiber gives a homomorphism
\[
\CH_i^G(X) \to \CH_i(X),
\]
which is independent of the choice of fiber.
The image is contained in the $G$-invariant subgroup $\CH_i(X)^G$.

\subsection*{Equivariant decomposition of the diagonal}

%a $G$-variety is a variety over $\bC$, with regular action of a finite abelian group $G$.
%For a module $M$, $M_{\mathrm{tf}}$ denotes $M/M_{\mathrm{tor}}$ where $M_{\mathrm{tor}}$ is the torsion part of $M$.

Let $X$ be a projective variety of dimension $n$ over a field $K$.
We recall,
$X$ is said to have
\emph{integral decomposition of the diagonal} if there exist a zero-dimensional cycle class $\zeta\in \mathrm{CH}_0(X)$ of degree $1$ and closed $D\subsetneq X$, such that the class of the diagonal $\Delta_X$ satisfies
\[
[\Delta_X]=([X]\times \zeta)+\gamma
\]
in $\mathrm{CH}_n(X\times X)$, for an
$n$-dimensional cycle class $\gamma$,
supported on $D\times X$.
This implies that
$\zeta=[x]$ in $\mathrm{CH}_0(X)$, for any smooth $K$-point $x$ in the complement of $D$ (by
restriction to $x\times X$).
Thus, if $X(K)$ is dense in $X$, the definition may be equivalently stated with the existence of a smooth $K$-point $x$, taking the place of $\zeta$ in the equality.

\begin{defi}
\label{defn:diag}
Let $X$ be an irreducible projective $G$-variety of dimension $n$.
We say that $X$ has
\emph{$G$-equivariant integral decomposition of the diagonal}
if there exist $\zeta\in \mathrm{CH}_0^G(X)$, with $\deg^G_0(\zeta)=1$, and
$G$-invariant closed $D\subsetneq X$,
such that
\begin{equation}
\label{eqn:diag}
[\Delta_X]=([X]\times \zeta)+\gamma
\end{equation}
in $\mathrm{CH}_n^G(X \times X)$,
for some $\gamma\in \mathrm{im}(
\mathrm{CH}^G_n(D\times X)\to \mathrm{CH}^G_n(X\times X))$.
\end{defi}

\begin{lemm}
\label{lem:affinebundle}
Let $K$ be a field,
$X$ a finite-type $K$-scheme, $\ell$ a positive integer, and $S\subsetneq \bA^\ell_K$ a closed subscheme.
The inclusion map
\[ X\times S\to X\times \bA^\ell_K \]
induces, by pushforward, the zero map on Chow groups.
\end{lemm}

\begin{proof}
Pullback gives an isomorphism $\CH_i(X)\to \CH_{i+\ell}(X\times \bA^\ell_K)$, with inverse given by intersection with $X\times \{v\}$ for any $K$-point $v\in \bA^\ell_K$.
So it suffices to show the vanishing of the composite
\[ \CH_{i+\ell}(X\times S)\to \CH_{i+\ell}(X\times \bA^\ell_K)\to \CH_i(X) \]
with the inverse isomorphism.
This is clear
if $K$ is infinite, by choosing
$v\in \bA^\ell_K\setminus S$, while
the case that $K$ is finite can be treated with a pair of points $v$, $v'$ valued in extension fields of $K$ of coprime degrees.
\end{proof}

\begin{theo}
\label{thm.idodtwist}
Let $X$ be an irreducible projective $G$-variety over $k$ with generically free $G$-action.
The following are equivalent.
\begin{itemize}
\item[(i)] For every torsor $T$, ${}^TX$ has integral decomposition of the diagonal.
\item[(ii)] For some quasi-projective $G$-variety $Z$ over $k$ with generically free versal $G$-action, ${}^{T_Z}X$ has integral decomposition of the diagonal.
\item[(iii)] $X$ has $G$-equivariant integral decomposition of the diagonal.
\end{itemize}
\end{theo}

\begin{proof}
The implication (i) $\Rightarrow$ (ii) is trivial.
We show (ii) $\Rightarrow$ (i).
Let $d$ denote the dimension of $Z$.
An integral decomposition of the diagonal determines, by spreading out, an expression in $\mathrm{CH}_{n+d}((X\times X\times Y)/G)$
\begin{equation}
\label{eqn.inXXY}
[(\Delta_X\times Y)/G]=\sum_i a_i [(X\times Y_i)/G]+\gamma,
\end{equation}
for some dense invariant open $Y\subset Z$,
without loss of generality contained in the nonsingular locus of $Z$, integers $a_i$,
invariant multisections $Y_i\subset X\times Y$, flat over $Y$, of the projection map $X\times Y\to Y$, of respective degrees $d_i$, with $\sum a_i d_i=1$, and class $\gamma$ of an invariant cycle supported on $D\times X\times Y$, for some $G$-invariant closed $D\subsetneq X$.
Given a torsor $T\colon L/K$, we take $G$-equivariant $\Spec(L)\to Y$,  say dominant over $G$-invariant closed $Y'\subset Y$ of codimension $c$.
The inclusion gives rise to a Gysin pullback, hence an analogous expression in $\mathrm{CH}_{n+d-c}((X\times X\times Y')/G)$.
We restrict over the generic point of $Y'/G$ and extend scalars to obtain an integral decomposition of the diagonal on ${}^TX$.

Taking $Z=V$, a $G$-representation of dimension $\ell$ with $G$-invariant open $U\subset V$ with free $G$-action, $\codim(V\setminus U)>n$, gives (iii) $\Rightarrow$ (ii).
We complete the proof by showing (i) $\Rightarrow$ (iii).
By (i), we have integral decomposition of the diagonal for the twist by torsor $T_V$, where we are taking $Z=V$.
Spreading out, we may suppose that we have an expression
\eqref{eqn.inXXY}.
Let $W=V\setminus Y$.
Taking closures, from \eqref{eqn.inXXY} we get
\[ [(\Delta_X\times V)/G]=\sum a_i[(X\times \overline{Y}_i)/G]+\gamma+\delta, \]
with $\gamma$ supported on $D\times X\times V$ and
$\delta$ supported on $X\times X\times W$.
Let $T=T_X$, the torsor $L/K$ with $L=k(X)$ and $K=k(X)^G$.
With base change by $\Spec(K)\to X/G$ and
an identification ${}^{T_X}V\cong\bA^\ell_K$,
yielding from $W$ closed $S\subsetneq \bA^\ell_K$,
we obtain from $\delta$ a cycle class,
supported on ${}^{T_X}X \times S$.
By Lemma \ref{lem:affinebundle} this is rationally equivalent to $0$ on ${}^{T_X}X\times \bA^\ell_K$.
Spreading out, we obtain a rational equivalence on $(X\times X\times V)/G$ between $\delta$ and $\gamma'$ for a class $\gamma'$ of an invariant cycle supported on $D'\times X\times V$, for some invariant closed $D'\subsetneq X$.
Thus, $X$ satisfies $G$-equivariant integral decomposition of the diagonal.
\end{proof}

\begin{rema}
\label{rem:choicezeta}
If the $G$-action on $X$ is generically free and versal, then Definition \ref{defn:diag} may be equivalently stated with $\zeta$ represented by an invariant $\widetilde{U}\subset X\times U$,
projecting birationally to $U$, with $\widetilde{U}\not\subset X^{\mathrm{sing}}\times U$.
Indeed, versality implies density of rational points on any twist ${}^TX$.
In the portion of the proof with $T=T_V$ and $Z=V$, over $K=k(V)^G$ we use the observation that integral decomposition of the diagonal may be equivalently stated with a smooth $K$-point in the role of $\zeta$.
We carry out the proof of (i) $\Rightarrow$ (iii) in Theorem~\ref{thm.idodtwist} with just a single index $i=1$, an equivariant section $Y_1$, and $a_1=1$, to get what is claimed, with $\widetilde{U}=(X\times U)\cap\overline{Y}_1$.
\end{rema}

As is well known (cf.\ \cite[Sect.\ 3]{saltmanretract} and \cite[Lemma 1.5]{CTP}), stably rational varieties admit an integral decomposition of the diagonal.
Here we give an analogous result in equivariant birational geometry.

\begin{theo}
\label{theo:stablelinearizable}
Let $G$ be a finite group and $X$ a smooth projective variety with a generically free $G$-action. If the action is stably linearizable then $X$ admits a $G$-equivariant integral decomposition of the diagonal.
\end{theo}

\begin{proof}
Theorem~\ref{thm.idodtwist} reduces the statement to stable rationality of twists of $X$.
\end{proof}

The following is analogous to \cite[Thm.\ 2.1]{Voisin15} and \cite[Thm.\ 2.3]{CTP}.

\begin{theo}
\label{theo:specialization}
Let $G$ be a finite group.
Let $\pi : \mathcal X \to B$ be a smooth proper family of irreducible $G$-varieties over a smooth curve $B$ such that $G$ acts on $\mathcal X$ and $\pi$ is $G$-invariant.
Suppose that the general fiber $\mathcal X_b$ has a $G$-equivariant integral decomposition of the diagonal.
Then a special fiber $\mathcal X_{b_0}$ has a  $G$-equivariant integral decomposition of the diagonal.
\end{theo}

\begin{proof}
This follows from \cite[Thm.\ 1.12]{CTP} combined with Theorem~\ref{thm.idodtwist}.
\end{proof}

\section{Cohomological invariants}
\label{sect:coho}

Cohomology, and in particular, the Leray spectral sequence associated to the morphism of stacks
$$
[X/G]\to BG,
$$
supplies equivariant birational invariants of $G$-actions. In this section, we investigate the behavior of these under passage to nonclosed fields, via the twisting construction in Section~\ref{sect:gen}.

\subsection*{Universal torsor obstruction and Amitsur invariants}
We recall the definition of the Amitsur invariant
$$
\Am^2(X,G):=\Pic(X)^G/\Pic([X/G])
$$
in equivariant geometry; see, e.g., \cite{BCDP}.
This group captures obstructions to the $G$-linearization of $G$-invariant classes in $\Pic(X)$.
When $X$ is a smooth projective irreducible $G$-variety, $\Am^2(X,G)$ can be expressed as the image of the
homomorphism
\[ \delta_2\colon \Pic(X)^G\to \rH^2(G,k^\times) \]
from the Leray spectral sequence.
Higher Amitsur invariants $\Am^d(X,G)$ ($d\ge 2$) are defined similarly, as the images of homomorphisms
\[ 
\delta_d\colon \rH^{d-2}(G,\Pic(X))\to \rH^d(G,k^\times).
\]
Assume that $\Pic(X)$ is $\bZ$-free and finitely generated. 
A similar spectral sequence gives rise to an exact sequence \cite[(3.1)]{KT-uni}
\begin{equation}
\label{eqn.TNSseq}
\rH^1_G(X,T_{\mathrm{NS}})\to 
\End(\Pic(X))^G\stackrel{\partial}\to \rH^2(G,T_{\mathrm{NS}}(k))\to \rH^2_G(X,T_{\mathrm{NS}}),
\end{equation}
where $T_{\mathrm{NS}}$ denotes the N\'eron-Severi torus,
and we have the {\em universal torsor obstruction}
\[ \beta(X\actsfromright G):=\partial(1_{\Pic(X)}). \]

The Amitsur invariants $\Am^d(X,G)$ ($d\ge 2$) and the vanishing of the universal torsor obstruction $\beta(X\actsfromright G)$ are stable $G$-birational invariants of $X$.
As well, $\beta(X\actsfromright G)=0$ implies $\Am^d(X,G)=0$ for all $d\ge 2$ \cite[Thm.\ 1.2]{STZ}.

\begin{prop}
\label{prop.Am2obstructsWV}
Let $X$ be a smooth projective irreducible $G$-variety, such that the
$G$-action is weakly versal.
Then:
\begin{itemize}
\item[(i)] $\Am^d(X,G)=0$, for all $d\ge 2$.
\item[(ii)] If $\Pic(X)$ is $\bZ$-free and finitely generated, then $\beta(X\actsfromright G)=0$.
\end{itemize}
\end{prop}

\begin{proof}
Assertion (i) is immediate from \cite[Thm.\ 5.1]{STZ}.
Assertion (ii) is \cite[Thm.\ 5.3]{KT-uni}, where the stated hypothesis that $X$ is rational can be replaced by the more general condition, that $\Pic(X)$ is $\bZ$-free and finitely generated.
\end{proof}

Let $X$ be a variety over an
extension field $K$ of $k$. 
We denote by $\overline{K}$ an algebraic closure and by $\mathfrak g_K$ the absolute Galois group of $K$.
Then there is a similar, classical,  invariant
$$
\Am^2(X):=\Pic(X_{\overline{K}})^{\mathfrak g_K}/\Pic(X).
$$
In case $X$ is a nonsingular geometrically irreducible projective variety over $K$, we identify $\Am^2(X)$ with the subgroup of the Brauer group $\Br(K)\simeq \rH^2(\mathfrak g_K,\overline{K}^\times)$, that arises as the image of $\Pic(X_{\overline{K}})^{\mathfrak g_K}$ from the analogous spectral sequence.

\subsection*{Negligible cohomology}

We recall the definition and basic properties, following \cite{GM}. A continuous homomorphism 
$$
\rho_K\colon \mathfrak g_K \to G
$$
yields a homomorphism in cohomology
$$
\rho_K^*\colon \rH^d(G,M)\to \rH^d(\mathfrak g_K,M), 
$$
for every $G$-module $M$. The group of {\em negligible} classes
$$
\rH^d(G,M)_{\mathrm{neg}}:=
\{ \alpha \mid \rho_K^*(\alpha)=0, \quad \forall\, K/k\} \subseteq \rH^d(G,M)
$$
consists of classes not detectable in Galois cohomology. In particular, this applies to the twisting construction associated with a quasi-projective $G$-variety $Z$ with generically free $G$-action, in which case there is a natural surjective homomorphism $\rho_Z\colon \mathfrak g_{k(Z)^G}\to G$, arising from the Galois extension $k(Z)/k(Z)^G$. 

The following result strengthens \cite[Prop.\ 2.1]{GM}:

\begin{prop}
\label{prop:gm}
A class $\alpha\in \rH^d(G,M)$ is negligible if and only if
$\rho_Z^*(\alpha)=0$, for some quasi-projective $G$-variety $Z$ over $k$ with generically free versal $G$-action.
\end{prop}

\begin{proof}
The forward implication is trivial.
For the reverse implication, given $\rho_Z^*(\alpha)=0$, there exists dense invariant open $W\subset Z$, such that the $G$-action on $W$ is free, $\rho_Z^*$ factors as
\[ \rH^d(G,M)\to \rH^d(W/G,M)\to \rH^d(\mathfrak g_{k(Z)^G},M), \]
and $\alpha$ maps to $0$ in the \'etale cohomology $\rH^d(W/G,M)$.
Given a torsor $T\colon L/K$, with corresponding $\rho_K\colon \mathfrak g_K \to G$, there exists $G$-equivariant $\Spec(L)\to W$.
Then $\rho_K^*$ also factors through $\rH^d(W/G,M)$, thus
$\rho_K^*(\alpha)=0$.
\end{proof}

\begin{prop}
\label{prop.reproduceAm2}
Let $X$ be a smooth projective irreducible $G$-variety with $\bZ$-free finitely generated $\Pic(X)$, let
$Z$ be a quasi-projective $G$-variety with generically free versal $G$-action.
Then the twisting construction by the associated torsor
$T_Z=k(Z)/k(Z)^G$ over $K=k(Z)^G$ gives rise to an isomorphism
\[
\Am^2(X,G)\eqto \Am^2({}^{T_Z} X).
\]
Furthermore, the homomorphism
\[
\rH^2(G,k^\times)\to \rH^2(\mathfrak g_K,\overline{K}^\times)
\]
is injective and restricts to the above isomorphism of Amitsur groups.
\end{prop}

\begin{proof}
The hypotheses imply $\rH^1(X,\cO_X)=0$ and thus 
$$
\Pic(X\times T)\cong \Pic(X)\oplus \Pic(T),
$$ 
for any finite-type $k$-scheme $T$.
A standard limit arguments for coherent sheaves \cite[(8.5.2)]{EGAIV}
identifies $\Pic({}^{T_Z}X)$ with
$$
\varinjlim \Pic((X\times W)/G),
$$
where the limit is over
nonempty invariant affine open $W\subset Z$.
A similar treatment of $\Pic(X_{k(Z)})$, in combination with the expression for Picard group of a product with $X$, gives an equivariant isomorphism
\begin{equation}
\label{eqn.PicXkZ}
\Pic(X)\eqto \Pic(X_{k(Z)}).
\end{equation}
Since group cohomology commutes with direct limits \cite[Prop.\ I.8]{SerreCG},
we may identify $\rH^2(G,k(Z)^\times)$ with
$\varinjlim \rH^2(G,k[W]^\times)$.
With this, we deduce that $\rH^2(G,k^\times)\to \rH^2(G,k(Z)^\times)$ is injective by combining the vanishing of $\Am^2(Z,G)$ (Proposition \ref{prop.Am2obstructsWV}) with
functoriality of the Leray spectral sequence and \cite[Prop.\ 2.5 (iv)]{ant}, which gives injectivity of $\Br([Z/G]\to \Br([W/G])$.
Now $\rH^2(G,k(Z)^\times)$ is identified with the subgroup of
$\Br(K)\simeq\rH^2(\mathfrak g_K,\overline{K}^\times)$,
of classes, trivialized by restriction to $k(Z)$;
cf., e.g., \cite[Thm.\ 1.3.5]{CTSkoro}.

We establish the isomorphism of Amitsur groups, first, in the case that
$Z=V$, a linear representation with invariant open $U\subset V$ on which $G$ acts freely and complement of $U$ in $V$ of codimension $\ge 2$.
Let $n$ denote the dimension of $X$, and $\ell$ the dimension of the representation $V$.
By \cite[Thm.\ 1]{EG98}, we have
\[ \Pic([X/G])\cong \CH^G_{n-1}(X)\cong \CH_{n+\ell-1}((X\times U)/G). \]
We have equivariant identifications of $\Pic(X)\cong \CH_{n-1}(X)$, first, with $\CH_{n+\ell-1}(X\times V)$, then with $\Pic(X\times U)\cong \CH_{n+\ell-1}(X\times U)$.
In the commutative diagram
\[
\xymatrix{
\Pic([X/G])\ar[r]\ar[d]^\sim &
\Pic(X)^G \ar[r]\ar[d]^\sim &
\rH^2(G,k^\times) \ar@{=}[d] \\
\Pic((X\times U)/G)\ar[r]\ar[d] & 
\Pic(X\times U)^G\ar[r]\ar[d]^\sim &
\rH^2(G,k^\times)\ar@{^{(}->}[d] \\
\Pic({}^{T_V}X)\ar[r] & 
\Pic(X_{k(V)})^G\ar[r] &
\rH^2(G,k(V)^\times)
}
\]
the bottom-left vertical map is surjective.
So we have the isomorphism of Amitsur groups, compatible with the injective map on $\rH^2$.

Now we treat the case of general $Z$.
Basic functoriality gives the homomorphism of Amitsur groups, compatible with homomorphism on $\rH^2$, and the latter is injective.
By the isomorphism \eqref{eqn.PicXkZ}, the homomorphism of
Amitsur groups is surjective.
Now suppose we have a line bundle $\cL$ on $X$ with $G$-invariant class in $\Pic(X)$, such that the class of $\cL_{k(Z)}$ on $X_{k(Z)}$
lifts to $\Pic({}^{T_Z}X)$, i.e., to
$\Pic((X\times W)/G)$ for some
nonempty invariant affine open $W\subset Z$.
By versality there exists an equivariant morphism
$\Spec(k(V))\to W$.
By a similar diagram as above, with
$\Pic((X\times W)/G)$ and $\Pic(X\times W)^G$ in the middle row,
and the established isomorphism $\Am^2(X,G)\cong \Am^2({}^{T_V}X)$,
we obtain the vanishing of the class of $\cL$ in $\Am^2(X,G)$.
\end{proof}

\begin{exam} 
An analogous statement for higher Amitsur groups fails.
There exist generically free regular actions of $G=\mathfrak{Q}_8$ on a del Pezzo surface $X$ of degree 2 with nontrivial $\mathrm{Am}^3(X,G)$ \cite{TZcohomological}.
If we take $Z=V$, faithful $2$-dimensional representation, then $k(Z)^G$ is a $C_2$-field, so
$\mathrm{Am}^3$ has to vanish for the twist ${}^{T_Z}X$.
Another example, with nontrivial $\mathrm{Am}^n(X,G)$ and vanishing $\Am^n({}^{T_Z}X)$, for all $n\ge 3$, can be found in \cite[Sect.\ 9]{STZ}. 
\end{exam}

\bibliographystyle{alpha} 
\bibliography{coarse}

\end{document}